\documentclass[ pdflatex]{sn-jnl}

\usepackage{amsmath,amssymb,amsthm}

\newtheorem{theorem}{Theorem}
\newtheorem{lemma}{Lemma}

\begin{document}

\title{On Carmichael Numbers Constructed from Consecutive Jacobsthal Numbers}


\author{\fnm{Hayat} \sur{Bensella}}\email{hbensella@usthb.dz / bensella.hayet@gmail.com}

\affil{\orgdiv{Faculty of Mathematics, LATN Laboratory}, \orgname{USTHB}, \orgaddress{\street{BP 32, El Alia}, \city{Bab Azzouar}, \postcode{16111}, \state{Algiers}, \country{Algeria}}}

\abstract{
We prove that no Carmichael number has the form
\[
N=(2^mJ_n+1)(2^mJ_{n+1}+1)(2^mJ_{n+2}+1),
\]
where
\[
J_n=\frac{2^n-(-1)^n}{3}
\]
is the Jacobsthal sequence, \(m,n\ge1\), and all three factors are prime.
Assuming that such an \(N\) is Carmichael, Korselt's criterion yields
divisibility conditions relating consecutive Jacobsthal numbers to powers
of two. Combining these conditions with the Binet formula and a finite
case analysis for \(n\le4\), we obtain a contradiction. Consequently,
this family contains no Carmichael numbers.
}


\maketitle

\section{Introduction}

Carmichael numbers are composite integers $N$ satisfying
\[
a^{N-1}\equiv 1\pmod N
\]
for every integer $a$ such that $\gcd(a,N)=1$.
By Korselt's criterion, a composite integer $N$ is a
Carmichael number if and only if $N$ is squarefree and
\[
p-1\mid N-1
\]
for every prime divisor $p$ of $N$.\\
As a particular case of the Lucas sequences of the first kind,
the Jacobsthal sequence $\{J_m\}_{m\geq 0}$ is defined by
$J_0=0$, $J_1=1$, and the recurrence relation
\[
J_m=J_{m-1}+2J_{m-2}, \qquad m\geq 2.
\]

Its Binet formula is given by
\begin{equation*}
J_m=\frac{2^m-(-1)^m}{3},
\label{eq:jacobsthal-binet}
\end{equation*}
where $2$ and $-1$ are the roots of the characteristic
equation
\[
P(x)=x^2-x-2.
\]

For further information on Jacobsthal numbers and their
properties, see \cite{Horadam2}.\\

The study of Carmichael numbers has a long history.
Alford, Granville, and Pomerance proved that there are
infinitely many Carmichael numbers \cite{Alford1994}.
Subsequent research has investigated their distribution
and their occurrence in special arithmetic forms.
Banks, Finch, Luca, Pomerance, and St\u{a}nic\u{a}
studied Sierpi\'nski and Carmichael numbers
\cite{Banks2015}, while Cilleruelo, Luca, and Pizarro
examined Carmichael numbers in sequences of the form
$(2^n k+1)_{n\geq 1}$ \cite{Cilleruelo2016}.
Other contributions have addressed the density of
related arithmetic sets \cite{Erdos1979}, the
distribution of primes in short intervals
\cite{Granville2015}, and the impossibility of certain
types of Carmichael numbers \cite{Wright2012}.
More recently, Alahmadi and Luca proved that no
Carmichael number can be of the form $2^n p+1$,
where $p$ is prime \cite{AlahmadiLuca2022}.

Motivated by these investigations, we consider integers
constructed from three consecutive terms of the
Jacobsthal sequence.

We study integers of the form
\[
N=(2^mJ_n+1)(2^mJ_{n+1}+1)(2^mJ_{n+2}+1),
\]
where $m,n\geq 1$.

Our main result establishes that, whenever all three
factors are prime, their product cannot be a
Carmichael number.

\medskip

\begin{theorem}\label{thm1}

Let
\[
N=(2^mJ_n+1)(2^mJ_{n+1}+1)(2^mJ_{n+2}+1),
\]
where
\[
J_k=\frac{2^k-(-1)^k}{3},
\qquad m,n\geq 1.
\]
If all three factors are prime, then $N$ is not a
Carmichael number.
    
\end{theorem}
\section{Preliminaries}

The Jacobsthal numbers satisfy
\[
J_1=1,\qquad J_2=1,\qquad J_{k+2}=J_{k+1}+2J_k.
\]

\begin{lemma}\label{lem:odd}
For every \(k\ge1\), \(J_k\) is odd.
\end{lemma}

\begin{proof}
Since \(2^k\) is even and \((-1)^k=\pm1\), the numerator
\(2^k-(-1)^k\) is odd. As \(3\) is odd, \(J_k\) is odd.
\end{proof}

\begin{lemma}\label{lem:gcd}
For every \(k\ge1\),
\[
\gcd(J_k,J_{k+1})=1
\qquad\text{and}\qquad
\gcd(J_k,J_{k+2})=1.
\]
\end{lemma}

\begin{proof}
The recurrence \(J_{k+2}=J_{k+1}+2J_k\) gives
\[
\gcd(J_k,J_{k+1})
=
\gcd(J_k,J_{k+1}+2J_k)
=
\gcd(J_k,J_{k+2}).
\]
Also,
\[
J_{k+1}=J_k+2J_{k-1},
\]
so
\[
\gcd(J_k,J_{k+1})
=
\gcd(J_k,2J_{k-1}).
\]
By Lemma~\ref{lem:odd}, \(J_k\) is odd, hence
\[
\gcd(J_k,J_{k+1})=\gcd(J_k,J_{k-1}).
\]
Iterating and using \(J_1=1\), we obtain
\(\gcd(J_k,J_{k+1})=1\). Therefore also
\(\gcd(J_k,J_{k+2})=1\).
\end{proof}

\begin{theorem}[Korselt's criterion] \cite{Korselt1899}
A composite integer $N$ is a Carmichael number if and only if
$N$ is squarefree and
\[
p-1\mid N-1
\]
for every prime divisor $p$ of $N$.
\end{theorem}

\section{Proof of Theorem \ref{thm1}}
Applying Korselt's criterion to the three prime factors of \(N\), we obtain
the following necessary divisibility conditions.
\begin{lemma}\label{lem:korselt}
Assume that
\[
N=(2^mJ_t+1)(2^mJ_{t+1}+1)(2^mJ_{t+2}+1)
\]
is a Carmichael number, where \(m,t\ge1\) and the three factors are prime.
Then
\[
J_t\mid (-1)^t2^m+2,
\]
\[
J_{t+1}\mid 3-(-1)^t2^m,
\]
\[
J_{t+2}\mid 1-(-1)^t2^{m-1}.
\]
\end{lemma}

\begin{proof}
Let
\[
p_1=2^mJ_t+1,\qquad
p_2=2^mJ_{t+1}+1,\qquad
p_3=2^mJ_{t+2}+1.
\]
Since \(N\) is Carmichael, it is squarefree, and Korselt's criterion gives
\[
p_i-1\mid N-1
\qquad(i=1,2,3).
\]

Set
\[
k_1=J_t,\qquad k_2=J_{t+1},\qquad k_3=J_{t+2}.
\]
Then
\[
N-1
=
2^m(k_1+k_2+k_3)
+2^{2m}(k_1k_2+k_1k_3+k_2k_3)
+2^{3m}k_1k_2k_3.
\]
By Lemma~\ref{lem:odd}, each \(k_i\) is odd. Hence
\(\gcd(2^m,k_i)=1\). Therefore \(k_i\mid N-1\) is equivalent to the
condition obtained after dividing by \(2^m\):
\[
k_1\mid (k_2+k_3)+2^mk_2k_3,
\]
\[
k_2\mid (k_1+k_3)+2^mk_1k_3,
\]
\[
k_3\mid (k_1+k_2)+2^mk_1k_2.
\]

We use the recurrence
\[
J_{t+2}=J_{t+1}+2J_t.
\]

First condition:
\\

Modulo \(J_t\), we have \(J_{t+2}\equiv J_{t+1}\). Hence
\[
(J_{t+1}+J_{t+2})+2^mJ_{t+1}J_{t+2}
\equiv
2J_{t+1}+2^mJ_{t+1}^2
=
J_{t+1}(2^mJ_{t+1}+2)
\pmod{J_t}.
\]
Thus
\[
J_t\mid J_{t+1}(2^mJ_{t+1}+2).
\]
By Lemma~\ref{lem:gcd}, \(\gcd(J_t,J_{t+1})=1\), so
\begin{equation}\label{eq:first}
J_t\mid 2^mJ_{t+1}+2.
\end{equation}

Second condition:\\

Modulo \(J_{t+1}\), the recurrence gives
\[
2J_t=J_{t+2}-J_{t+1}\equiv J_{t+2}.
\]

Multiplying the second divisibility condition by \(2\), which preserves
divisibility, gives
\[
J_{t+1}\mid
2(J_t+J_{t+2})+2^{m+1}J_tJ_{t+2}.
\]
Using \(2J_t\equiv J_{t+2}\pmod{J_{t+1}}\), we obtain

\[
2(J_t+J_{t+2})+2^{m+1}J_tJ_{t+2}
\equiv
3J_{t+2}+2^mJ_{t+2}^2
=
J_{t+2}(2^mJ_{t+2}+3)
\pmod{J_{t+1}}.
\]
Hence
\[
J_{t+1}\mid J_{t+2}(2^mJ_{t+2}+3).
\]
By Lemma~\ref{lem:gcd}, \(\gcd(J_{t+1},J_{t+2})=1\), so
\begin{equation}\label{eq:second}
J_{t+1}\mid 2^mJ_{t+2}+3.
\end{equation}

Third condition:\\

Modulo \(J_{t+2}\), the recurrence gives
\[
J_{t+1}=J_{t+2}-2J_t\equiv-2J_t.
\]
Therefore
\[
(J_t+J_{t+1})+2^mJ_tJ_{t+1}
\equiv
J_t-2J_t+2^mJ_t(-2J_t)
=
-J_t(1+2^{m+1}J_t)
\pmod{J_{t+2}}.
\]
Thus
\[
J_{t+2}\mid J_t(2^{m+1}J_t+1).
\]
By Lemma~\ref{lem:gcd}, \(\gcd(J_t,J_{t+2})=1\), so
\begin{equation}\label{eq:third}
J_{t+2}\mid 2^{m+1}J_t+1.
\end{equation}

Binet translation:
\\
Let \(\varepsilon=(-1)^t\). Then
\[
3J_t=2^t-\varepsilon,\qquad
3J_{t+1}=2^{t+1}+\varepsilon,\qquad
3J_{t+2}=2^{t+2}-\varepsilon.
\]

From \eqref{eq:first},
\[
J_t\mid 2^mJ_{t+1}+2.
\]
Multiplying by \(3\),
\[
2^t-\varepsilon\mid 2^{m+t+1}+\varepsilon2^m+6.
\]
Since \(2^t\equiv\varepsilon\pmod{2^t-\varepsilon}\),
\[
2^{m+t+1}+\varepsilon2^m+6
\equiv
\varepsilon2^{m+1}+\varepsilon2^m+6
=
3(\varepsilon2^m+2).
\]
Hence
\[
2^t-\varepsilon\mid 3(\varepsilon2^m+2).
\]
Since \(2^t-\varepsilon=3J_t\), dividing by \(3\) gives
\begin{equation}\label{eq:binet1}
J_t\mid \varepsilon2^m+2.
\end{equation}

From \eqref{eq:second},
\[
J_{t+1}\mid 2^mJ_{t+2}+3.
\]
Multiplying by \(3\),
\[
2^{t+1}+\varepsilon\mid 2^{m+t+2}-\varepsilon2^m+9.
\]
Since \(2^{t+1}\equiv-\varepsilon\), we have
\(2^{t+2}\equiv-2\varepsilon\). Therefore
\[
2^{m+t+2}-\varepsilon2^m+9
\equiv
-3\varepsilon2^m+9
=
3(3-\varepsilon2^m).
\]
Thus
\[
2^{t+1}+\varepsilon\mid 3(3-\varepsilon2^m).
\]
Since \(2^{t+1}+\varepsilon=3J_{t+1}\), dividing by \(3\) gives
\begin{equation}\label{eq:binet2}
J_{t+1}\mid 3-\varepsilon2^m.
\end{equation}

From \eqref{eq:third},
\[
J_{t+2}\mid 2^{m+1}J_t+1.
\]
Multiplying by \(3\),
\[
2^{t+2}-\varepsilon\mid 2^{m+t+1}-\varepsilon2^{m+1}+3.
\]
Since \(2^{t+2}\equiv\varepsilon\), we have
\[
2^{m+t+1}-\varepsilon2^{m+1}+3
\equiv
\varepsilon2^{m-1}-\varepsilon2^{m+1}+3
=
3(1-\varepsilon2^{m-1}).
\]
Thus
\[
2^{t+2}-\varepsilon\mid 3(1-\varepsilon2^{m-1}).
\]
Since \(2^{t+2}-\varepsilon=3J_{t+2}\), dividing by \(3\) gives
\begin{equation}\label{eq:binet3}
J_{t+2}\mid 1-\varepsilon2^{m-1}.
\end{equation}

This proves the lemma.
\end{proof}

We are now ready to prove Theorem \ref{thm1}.

Suppose, for contradiction, that \(N\) is a Carmichael number.
By Lemma~\ref{lem:korselt}, with \(t=n\) and \(\varepsilon=(-1)^n\), we have
\begin{align}
J_n&\mid \varepsilon2^m+2, \label{eq:A}\\
J_{n+1}&\mid 3-\varepsilon2^m, \label{eq:B}\\
J_{n+2}&\mid 1-\varepsilon2^{m-1}. \label{eq:C}
\end{align}

We first show that \(n\le4\). Assume \(n\ge5\), and write
\[
m=qn+r,
\qquad
0\le r<n.
\]

\subsection{First divisibility}

By the first divisibility condition \eqref{eq:A},
\[
J_n\mid \varepsilon 2^m+2,
\qquad \varepsilon=(-1)^n.
\]
Since
\[
3J_n=2^n-\varepsilon,
\]
we obtain
\[
2^n-\varepsilon\mid 3(\varepsilon 2^m+2).
\]
Write
\[
m=qn+r,\qquad q\ge0,\quad 0\le r<n,
\]
and set
\[
\sigma=\varepsilon^{q+1}\in\{-1,1\}.
\]
As \(2^n\equiv\varepsilon\pmod{2^n-\varepsilon}\), we have
\[
2^m\equiv \varepsilon^q2^r
\pmod{2^n-\varepsilon}.
\]
Consequently,
\[
2^n-\varepsilon\mid 3(\sigma2^r+2).
\]
Furthermore,
\[
\begin{aligned}
\lvert 3(\sigma2^r+2)\rvert
&\le 3(2^r+2)\\
&\le 3(2^{n-1}+2)\\
&<2(2^n-\varepsilon),
\end{aligned}
\]
for \(n\ge5\). Hence there exists
\(\delta\in\{-1,0,1\}\) such that
\[
3(\sigma2^r+2)=\delta(2^n-\varepsilon).
\tag{1}
\]

If \(r=0\), then
\[
3(\sigma+2)\in\{3,9\},
\]
whereas \(2^n-\varepsilon\ge31\). This contradicts (1).
Therefore \(r\ge1\).

We next prove that \(\delta\ne\pm1\). Suppose first that
\(\delta=1\). Then
\[
3(\sigma2^r+2)=2^n-\varepsilon.
\]
If \(\sigma=1\), we obtain
\[
3\cdot2^r+6=2^n-\varepsilon,
\]
and thus
\[
2^r(2^{n-r}-3)=6+\varepsilon.
\]
This is impossible because the left-hand side is even, whereas
\(6+\varepsilon\in\{5,7\}\) is odd.

If \(\sigma=-1\), then
\[
6-3\cdot2^r=2^n-\varepsilon,
\]
which gives
\[
2^r(3+2^{n-r})=6+\varepsilon.
\]
Again, the left-hand side is even and the right-hand side is odd,
a contradiction.

Suppose now that \(\delta=-1\). Equation (1) becomes
\[
3(\sigma2^r+2)=-(2^n-\varepsilon).
\]
If \(\sigma=1\), the left-hand side is positive and the right-hand
side is negative, which is impossible.

If \(\sigma=-1\), we have
\[
6-3\cdot2^r=-2^n+\varepsilon,
\]
so
\[
2^r(3-2^{n-r})=6-\varepsilon.
\]
This is impossible because the left-hand side is even, whereas
\(6-\varepsilon\in\{5,7\}\) is odd.

Therefore \(\delta=0\). Equation (1) reduces to
\[
\sigma2^r+2=0.
\]
Since \(r\ge1\) and \(\sigma\in\{-1,1\}\), it follows that
\[
r=1,\qquad \sigma=-1.
\]
Finally, \(\sigma=\varepsilon^{q+1}=-1\) implies
\(\varepsilon=-1\). Hence \(n\) is odd.

\subsection{Second divisibility}

Because \(n\) is odd, \(\varepsilon=-1\). Thus \eqref{eq:B} becomes
\[
J_{n+1}\mid 3+2^m.
\]
Since \(n\) is odd,
\[
3J_{n+1}=2^{n+1}-1.
\]
Multiplying by \(3\), we get
\[
2^{n+1}-1\mid 3(3+2^m).
\]

Write
\[
m=q'(n+1)+s,
\qquad
0\le s\le n.
\]
Since \(2^{n+1}\equiv1\pmod{2^{n+1}-1}\), we have
\(2^m\equiv2^s\). Hence
\begin{equation}\label{eq:E}
2^{n+1}-1\mid 3(3+2^s).
\end{equation}

Now
\[
0<3(3+2^s)\le3(3+2^n)<2(2^{n+1}-1)
\]
for \(n\ge5\). Therefore the quotient in \eqref{eq:E} is a positive
integer strictly smaller than \(2\), hence equal to \(1\). Thus
\[
3(3+2^s)=2^{n+1}-1,
\]
or equivalently,
\begin{equation}\label{eq:F}
2^{n+1}=3\cdot2^s+10.
\end{equation}

We show that \eqref{eq:F} is impossible.

\begin{itemize}
\item If \(s=0\), then \(2^{n+1}=13\), impossible.
\item If \(s=1\), then \(2^{n+1}=16\), so \(n=3\), contradicting \(n\ge5\).
\item If \(s\ge2\), then
\[
3\cdot2^s+10\equiv2\pmod4,
\]
whereas
\[
2^{n+1}\equiv0\pmod4,
\]
impossible.
\end{itemize}

Therefore \(n\le4\).

\subsection{Small cases}

It remains to check \(n=1,2,3,4\).

\subsubsection*{Case \(n=1\)}

We have
\[
J_1=1,\qquad J_2=1.
\]
Hence the first two factors of \(N\) are equal:
\[
2^mJ_1+1=2^mJ_2+1.
\]
Thus \(N\) is not squarefree, contradicting the fact that every Carmichael
number is squarefree.

\subsubsection*{Case \(n=2\)}

Here \(\varepsilon=1\), and \eqref{eq:B} gives
\[
J_3=3\mid 3-2^m.
\]
But
\[
3-2^m\equiv-2^m\not\equiv0\pmod3,
\]
since \(3\nmid2^m\). Contradiction.

\subsubsection*{Case \(n=3\)}

Here \(\varepsilon=-1\). From \eqref{eq:A},
\[
J_3=3\mid -2^m+2,
\]
so
\[
2^m\equiv2\pmod3.
\]
Hence \(m\) is odd.

From \eqref{eq:C},
\[
J_5=11\mid 1+2^{m-1}.
\]
Thus
\[
2^{m-1}\equiv-1\pmod{11}.
\]
Since \(\operatorname{ord}_{11}(2)=10\) and
\(2^5\equiv-1\pmod{11}\), we get
\[
m-1\equiv5\pmod{10}.
\]
Therefore
\[
m\equiv6\pmod{10},
\]
so \(m\) is even, contradicting that \(m\) is odd.

\subsubsection*{Case \(n=4\)}

Here \(\varepsilon=1\). From \eqref{eq:A},
\[
J_4=5\mid 2^m+2,
\]
so
\[
2^m\equiv3\pmod5.
\]
Since \(\operatorname{ord}_5(2)=4\) and
\(2^3\equiv3\pmod5\), we get
\[
m\equiv3\pmod4.
\]
In particular, \(m\) is odd.

From \eqref{eq:B},
\[
J_5=11\mid 3-2^m,
\]
so
\[
2^m\equiv3\pmod{11}.
\]
Since \(\operatorname{ord}_{11}(2)=10\) and
\(2^8\equiv3\pmod{11}\), we get
\[
m\equiv8\pmod{10}.
\]
In particular, \(m\) is even, contradicting that \(m\) is odd.

All cases lead to a contradiction. Therefore no integer \(N\) of the
stated form, with \(m,n\ge1\) and all three factors prime, is a Carmichael
number.
\section*{Declarations}
The author declares that there are no conflicts of interest.

\end{document}